\documentclass[12pt]{article}
\usepackage{amsmath}
\usepackage{amsthm}
\usepackage{amssymb}
\usepackage{latexsym}
\usepackage{color}
\usepackage{graphicx}

\newcommand{\R}{\mbox{\boldmath $R$}}

\def \P {\mbox{\boldmath $P$}}
\newcommand{\T}{\mbox{\boldmath $T$}}
\newcommand{\TP}{\mbox{\boldmath $T\!P$}}
\newcommand{\ga}{\varGamma}

\newcommand{\la}{\Lambda}
\def \phi {\varphi}
\def \Phi {\varPhi}

\newcommand{\wt}{\widetilde}

\newcommand{\Div}{\operatorname{Div}}
\def \div {\operatorname{div}}
\newcommand{\ord}{\operatorname{ord}}
\newcommand{\Prin}{\operatorname{Prin}}
\newcommand{\rank}{\operatorname{rank}}
\newcommand{\Rat}{\operatorname{Rat}}
\newcommand{\supp}{\operatorname{supp}}

\newcommand{\nc}{\operatorname{NC}}
\newcommand{\tconv}{\operatorname{tconv}}
\newcommand{\algdim}{\operatorname{algdim}}
\newcommand{\geomdim}{\operatorname{geomdim}}

\newcommand{\set}[2]{\{ \, {#1} \mid {#2} \, \}}
\newcommand{\sset}[3]{\left\{ \, {#1} \mathrel{} \middle| \mathrel{} \begin{array}{l}{#2} \\ {#3} \end{array} \, \right\}}
\newcommand{\ratg}[2]{{#1}_{\geq {#1}({#2})}}

\newcommand{\q}[1]{\text{``}{#1}\text{''}}

\newtheorem{thm}{Theorem}[section]
\newtheorem{lem}[thm]{Lemma}
\newtheorem{prop}[thm]{Proposition}
\newtheorem{cor}[thm]{Corollary}

\theoremstyle{definition}

\newtheorem{rem}[thm]{Remark}

\title{Divisorial condition for the stable gonality of tropical curves}
\author{Yuki Kageyama}
\date{}

\begin{document}

\maketitle

\begin{abstract}
Let $d$ be a positive integer. There are several versions of $d$-gonality for tropical curves, stable $d$-gonality and divisorial $d$-gonality, which are both inspired by $d$-gonality for compact Riemann surfaces. However, that conditions are not equivalent. We have a condition of divisors equivalent to stable $d$-gonality for tropical curves.
\end{abstract}

\section{Introduction}

A tropical curve is a metric graph with possibly unbounded edges, which is an analogue of an algebraic curve or a compact Riemann surface. Divisors, ranks of divisors, and linear systems on tropical curves have been studied recently by other authors. A Riemann--Roch theorem for graphs was proven by Baker and Norine (\cite{BN07}), and one for tropical curves was established independently by Gathmann and Kerber (\cite{GK08}) and by Mikhalkin and Zharkov (\cite{MZ08}).

Let $d$ be a positive integer. A compact Riemann surface $X$ is $d$-\textit{gonal} if there exists a holomorphic map $X \to \P ^1$ of degree $d$. This condition is equivalent to the existence of a linear system on $X$ of degree $d$ and dimension one. A tropical curve $\ga$ is \textit{stably $d$-gonal} if there exists a finite harmonic morphism from some tropical modification of $\ga$ to some tree. The moduli space of genus $g$ tropical curves of stably $d$-gonal has the classical dimension $\min \{2g+2d-5,3g-3\}$ (\cite{CD18}). If $\ga$ is stably $d$-gonal, then $\ga$ has a linear system of degree $d$ and rank one. However, the converse does not hold.

For a finitely generated linear system $\la$ on $\ga$, we define the \textit{geometrical dimension} of $\la$. We prove the following:

\begin{thm}[$=$Corollary \ref{mthm}]
Let $\ga$ be a tropical curve, and let $d$ be a positive integer. Then, the followings are equivalent:
\begin{enumerate}
	\item[$(1)$] $\ga$ is stably $d$-gonal.
	\item[$(2)$] $\ga$ has a finitely generated liner system of degree $d$, rank one, and geometrical dimension one.
\end{enumerate}
\end{thm}








\section{Preliminaries}

\subsection{Tropically convex sets}

The \textit{tropical semifield} is the set $\T$ with two tropical operations:
	\[ \q{a+b} := \max \{ a,b \} \quad \text{and} \quad \q{a \cdot b} := a+b. \]
The set $\T^{n+1}$ is a semivector space over the tropical semifield. A subset $S$ of $\T^{n}$ is \textit{tropically convex} if the set $S$ contains the point $\q{a \cdot x + b \cdot y}$ for all $x,y \in S$ and all $a,b \in \T$. The \textit{tropical convex hull} of a given subset $V \subset \T^{n+1}$, denoted by $\tconv(V)$ is the smallest tropically convex subset of $\T^{n+1}$ which contains $V$. We identify the tropically convex set $S$ with its image in the $n$-dimensional \textit{tropical projective space} $\TP^{n}$. Tropically convex sets are contractible spaces (see \cite[Theorem 2]{DS04}). A \textit{tropical line segment} between two points $x$ and $y$ in $\TP^{n}$ is the tropical convex hull of $\{x,y\}$. This is the concatenation of at most $n$ ordinary line segments. Further, the slope of each line segment is a zero-one vector (see \cite[Proposition 3]{DS04}).

\subsection{Tropical curves}

A \textit{topological graph} is a topological space obtained by gluing a finite, disjoint union of closed intervals along an equivalence relation on the boundary points. If a topological graph $\ga$ is connected and the intervals from which it is glued are prescribed with a positive length, then $\ga$ becomes a compact metric space with shortest-path metric. Such a metric space is called a \textit{metric graph}. A \textit{tropical curve} is obtained from a metric graph by attaching finitely many intervals $[0,\infty]$ in such a way that each $0 \in [0,\infty]$ gets identified with a point of the metric graph. Each point $\infty \in [0,\infty]$ of a tropical curve is called a \textit{point at infinity}, and the others are \textit{finite points}.

Every point $x$ in a tropical curve has a neighborhood homeomorphic to a finite union of half-open intervals $[0,1)$ glued along $0$. Each of these intervals is called a \textit{half-edge} emanating from $x$, and their number is the \textit{valency} of $x$. We identify two half-edges emanating from $x$ if one is contained in the other.

\subsection{Divisors on tropical curves}

A \textit{divisor} on a tropical curve $\ga$ is an element of the free abelian group generated by the points of $\ga$. The group of divisors on $\ga$ is denoted by $\Div(\ga)$. For a divisor $D$, the coefficient of a point $x$ is denoted by $D(x)$. We define the \textit{degree} of $D$ by $\deg(D):=\sum_{x \in \ga} D(x)$. A divisor $D$ on $\ga$ is called \textit{effective}, and we write $D\geq 0$, if all the coefficients $D(x)$ are non-negative. The \textit{support} of $D$ is defined to be $\supp(D):=\set{x \in \ga}{D(x) \neq 0}$. A \textit{rational function} on $\ga$ is a function $f:\ga \to \R \cup \{\pm \infty \}$ which is continuous piecewise integral affine with a finite number of pieces such that $f$ can take on the values $\pm \infty$ only points at infinity, or $f \equiv -\infty$. The set of rational functions on $\ga$ is denoted by $\Rat(\ga)$. Let $f \in \Rat(\ga)^{\ast}$. The \textit{order} of $f$ at a point $x$ in $\ga$ is the sum of the outgoing slopes of $f$ at $x$, denoted by $\ord_{x}(f)$. The \textit{principal divisor} associated to $f$ is defined by
	\[ \div(f):= \sum_{x \in \ga} \ord_{x}(f) \cdot x. \]
The set of all principal divisors on $\ga$, denoted by $\Prin(\ga)$, becomes a subgroup of $\Div(\ga)$. It is easy to check that the degree of every principal divisor is zero. Two divisors $D$ and $D'$ on $\ga$ are called \textit{linearly equivalent}, and we write $D \sim D'$, if $D-D'$ is a principal divisor. The \textit{complete linear system} of a divisor $D$ on $\ga$, denoted by $|D|$, is the set of all effective divisors linearly equivalent to $D$. We define the subset $L(D)$ of $\Rat(\ga)$ as 
	\[ L(D):=\set{f \in \Rat(\ga)^{\ast}}{D+\div(f)\geq 0}\cup \{ -\infty\}.\]
The subset $L(D)$ becomes a finitely generated as a semivector space (see \cite[Lemma 4 and Corollary 9]{HMY12}). The complete linear system $|D|$ is identified with the projective space $\P L(D)$.


A \textit{linear system} is a subset $\la$ of a complete linear system $|D|$ such that
	\[ L_{\la}(D):=\set{f \in L(D)^{\ast}}{D+\div(f)\in \la} \cup \{ -\infty \} \]
is a subspace of $L(D)$. A linear system $\la$ is \textit{finitely generated} if $L_{\la}(D)$ is a finitely generated as a semivector space. Every finitely generated linear system has a unique minimal generating set of $L_{\la}(D)$ up to scalar multiplication (see \cite[Proposition 8]{HMY12}). Let $\algdim L_{\la}(D)$ denote the number of a minimal generating set of $L_{\la}(D)$, and let $\algdim \la:=\algdim L_{\la}(D)-1$.

A finitely generated linear system $\la$ has a finite polyhedral complex structure, and let $P(\la)$ denote this polyhedral complex. Let $\{f_{0},\ldots,f_{n}\}$ be a minimal generating set of $L_{\la}(D)$. We define a rational map $\Phi_{\la}:\ga\to\TP^{n}$ as $\Phi_{\la}(x):=(f_{0}(x):\cdots:f_{n}(x))$. Then, the polyhedral complex $P(\la)$ is homeomorphic to the tropical convex hull of the image of $\Phi_{\la}$ (see \cite[Theorem 27]{HMY12}). Let $\geomdim \la$ denote the dimension of $P(\la)$.

The \textit{rank} of a linear system $\la$ is defined as follows: If $\la$ is nonempty, then we define
	\[ \rank \la:=\max\sset{d}{\text{For any effective divisor $E$ of degree $d$, there}}{\text{exists $D' \in \la$ such that $D'-E$ is effective}}. \]
Otherwise, $\rank \la:=-1$.

\begin{rem}
For a finitely generated linear system $\la$, we have
	\[ \rank \la \leq \geomdim \la \leq \algdim \la . \]
\end{rem}

\subsection{Harmonic morphisms}

A continuous map $\phi$ from a tropical curve $\ga$ to a tropical curve $\ga'$ is a \textit{morphism} if linear with integral slopes outside a finite number of points. A morphism $\phi$ is \textit{finite} if $\phi^{-1}(\{x'\})$ is a finite set for any point $x'$ of $\ga'$. For a half-edge $h$ of a point $x \in \ga$, let $\deg_{h}(\phi)$ denote the slope of $\phi$ along $h$. A morphism $\phi$ is \textit{harmonic at a point $x \in \ga$} if for a half-edge $h'$ of $\phi(x)$, the integer
	\[ \sum_{\substack{x \in h \\ h \mapsto h'}}\deg_{h}(\phi) \]
is independent of the choice of $h'$. When $\phi$ is harmonic at $x \in \ga$, this sum is denoted by $\deg_{x}(\phi)$. A morphism $\phi$ is \textit{harmonic} if $\phi$ is harmonic at every point $x \in \ga$. A harmonic morphism $\phi$ has a well-defined degree, defined as
	\[ \deg(\phi):=\sum_{\substack{x \in \ga \\ \phi(x)=x'}}\deg_{x}(\phi) \]
for any $x' \in \ga'$.

Let $\phi:\ga \to \ga'$ be a harmonic morphism of tropical curves, and let $f$ and $f'$ be rational functions on $\ga$ and $\ga'$, respectively. The \textit{push-forward} of $f$ is the function $\phi_{\ast}f:\ga' \to \R \cup \{\pm \infty\}$ defined by
	\[ \phi_{\ast}f(x'):=\sum_{\substack{x \in \ga \\ \phi(x)=x'}} \deg_{x}(\phi) \cdot f(x). \]
The \textit{pull-back} of $f'$ is the function $\phi^{\ast}f':\ga \to \R \cup \{\pm \infty\}$ defined by $\phi^{\ast}f':=f' \circ \phi$. One can check that $\phi_{\ast}f$ and $\phi^{\ast}f'$ are again rational functions. We define the \textit{push-forward homomorphism} on divisors $\phi_{\ast}:\Div(\ga) \to \Div(\ga')$ by
	\[ \phi_{\ast}(D):=\sum_{x \in \ga} D(x) \cdot \phi(x). \]
The \textit{pull-back homomorphism} on divisors $\phi^{\ast} : \Div(\ga') \to \Div(\ga)$ is defined to be
	\[ \phi^{\ast} (D'):=\sum_{x \in \ga} \deg_{x}(\phi) \cdot D'(\phi(x)) \cdot x. \]
One can check that $\deg(\phi_{\ast}(D))=\deg(D)$ and $\deg(\phi^{\ast}(D'))=\deg(\phi) \cdot \deg(D')$ for any divisors $D$ on $\ga$ and $D'$ on $\ga'$.

\begin{prop}[cf. {\cite[Proposition 4.2]{BN09}}] \label{pu}
Let $\phi:\ga \to \ga'$ be a harmonic morphism of tropical curves. Then, we have
	\[ \phi_{\ast} (\div(f))=\div(\phi_{\ast} f) \quad \text{and} \quad \phi^{\ast} (\div(f'))=\div(\phi^{\ast} f') \]
for any rational functions $f$ on $\ga$ and $f'$ on $\ga'$.
\end{prop}

A \textit{tropical modification} of a tropical curve $\ga$ is a tropical curve $\wt{\ga}$ obtained from $\ga$ by grafting a finite number of trees on to finite points of $\ga$. Then, there exists a natural retraction morphism $\pi :\wt{\ga} \to \ga$ which is the identity on $\ga$ and contracts each tree to the grafted point. The morphism $\pi$ is a (non-finite) harmonic morphism of degree one. A tropical curve $\ga$ is \textit{stably $d$-gonal} if there exist tropical modification $\wt{\ga}$, a tree $T$, and a finite harmonic morphism $\wt{\phi}:\wt{\ga} \to T$ of degree $d$.

\section{Main results}

Let $\ga$ be a tropical curve. In this section, we will consider linear systems on $\ga$ satisfying the following condition:

$(\ast)$ $\la$ is a finitely generated linear system on $\ga$ of degree $d$, $\rank \la=\geomdim \la=1$.

Let $\wt{\ga}$ be a modification of $\ga$, let $\pi:\wt{\ga}\to \ga$ be the retraction morphism, let $\wt{\phi}:\wt{\ga}\to T$ be a finite harmonic morphism to a tree of degree $d$, and let $D$ be a divisor on $T$ of degree one. Then,
	\[ \pi_{\ast}(\wt{\phi}^{\ast}|D|):=\set{\pi_{\ast}(\wt{\phi}^{\ast}E) \in |\pi_{\ast}(\wt{\phi}^{\ast}D)|}{E \in |D|} \]
is a liner system on $\ga$ which satisfies $(\ast)$. Conversely, we will see that a linear system $\la$ on $\ga$ satisfying $(\ast)$ determines a modification $\wt{\ga}$ of $\ga$, a tree $T$, and a finite harmonic morphism $\wt{\phi}:\wt{\ga} \to T$ of degree $d$.

\begin{lem}
Let $\la$ satisfy $(\ast)$. Then, the image of the rational map $\Phi_{\la}$ is a tree, and this coincides with the its tropical convex hull.
\end{lem}

\begin{proof}
Since every tropically convex set is contractible, $\tconv(\Phi_{\la}(\ga))$ is a tree. Since $\Phi_{\la}(\ga)$ is connected, $\Phi_{\la}(\ga)$ is a subtree of $\tconv(\Phi_{\la}(\ga))$. Let $s$ be a tropical line segment between two points of $\Phi_{\la}(\ga)$. This is the unique simple path on $\tconv(\Phi_{\la}(\ga))$ between them, and this must coincide with the unique simple path on $\Phi_{\la}(\ga)$ between them. Therefore, $s$ is contained in $\Phi_{\la}(\ga)$ and thus, $\Phi_{\la}(\ga)$ is tropically convex.
\end{proof}

Let $\la$ satisfy $(\ast)$. Then, $\Phi_{\la}(\ga)$ has a tropical curve structure, where the distance of two points $(f_{0}(x):\cdots:f_{n}(x))$ and $(f_{0}(y):\cdots:f_{n}(y))$ in $\Phi_{\la}(\ga)$ is given by
	\[ \max\set{|f_{i}(x)+f_{j}(y)-f_{j}(x)-f_{i}(y)|}{0 \leq i<j \leq n} \]
(c.f. \cite{DS04}). Further, $\Phi_{\la}:\ga \to \Phi_{\la}(\ga)$ becomes a finite morphism. For each $x \in \ga$, let $D_{x}$ denote the element of $\la$ corresponding to the point $\Phi_{\la}(x) \in \tconv(\Phi_{\la}(\ga))$ under the isomorphism $P(\la) \simeq \tconv(\Phi_{\la}(\ga))$. Let $D \in {\la}$, and let $\{ f_{0},\ldots ,f_{n}\}$ be a minimal generating set of $L_{\la}(D)$. Then, $D_{x}$ is given by
	\[ D_{x}=D+ \div(\q{\sum_{i}\frac{f_{i}}{f_{i}(x)}}) \]
(see the proof of \cite[Theorem 27]{HMY12}).

For a rational function $f$ on $\ga$, and for $a \in \R \cup \{\pm\infty\}$, we define a rational function $f_{\geq a}$ on $\ga$ as
	\[ f_{\geq a}(x):=
	\begin{cases}
		\max \{ f(x),a \} & \text{if $a \in \R \cup \{ - \infty \}$,} \\
		0& \text{if $a=\infty$,}
	\end{cases} \]
for each $x \in \ga$.

\begin{lem} \label{base}
Let $\la$ satisfy $(\ast)$, let $x,y \in \ga$, and let $f \in L_{\la}(D_{x})^{\ast}$ such that $D_{y}=D_{x}+\div(f)$. Then, $f$ has the minimum (resp.\ maximum) value at $y$ (resp.\ $x$). Further, for $z \in \ga$ such that $\Phi_{\la}(z)$ is contained in the tropical line segment between $\Phi_{\la}(x)$ and $\Phi_{\la}(y)$, we have $D_{z}=D_{x}+\div(\ratg{f}{z})$.
\end{lem}

\begin{proof}
Let $\{ f_{0},\ldots ,f_{n}\}$ be a minimal generating set of $L_{\la}(D_{x})$. We may assume $\q{f_{0}(x)/f_{0}(y)}\leq \cdots \leq \q{f_{n}(x)/f_{n}(y)}$, and $f=\q{\sum_{i}f_{i}/f_{i}(y)}$. Let $g:=\q{\sum_{i}f_{i}/f_{i}(x)}$. Then, $g$ is the constant $0$ since $D_{x}=D_{x}+\div(g)$, and $g(x)=0$. Therefore, we have $0=g(y)=\q{\sum_{i}f_{i}(y)/f_{i}(x)} =\q{f_{0}(y)/f_{0}(x)}$, and $\q{f_{i}(y)/f_{i}(x)} \leq 0$ for each $i$. Then, $f$ has the minimum value at $y$ since $\q{f+g=f}$. Similarly, $f$ has the maximum value at $x$. We define a continuous map $u:[f(y),f(x)] \to \Phi_{\la}(\ga)$ as follows: For $t \in [f(y),f(x)]$, let $u(t)$ be the point of $\tconv(\Phi_{\la}(\ga))=\Phi_{\la}(\ga)$ corresponds to the element $D_{x}+\div(f_{\geq t}) \in \la$. Note that we have $u(f(y))=\Phi_{\la}(y)$ and $u(f(x))=\Phi_{\la}(x)$. Since $u$ is injective, the image of $u$ is a simple path between $\Phi_{\la}(y)$ and $\Phi_{\la}(x)$. This path coincides with the tropical line segment of them. We will show $u(f(z))=\Phi_{\la}(z)$.

Let $\q{\sum_{i}a_{i}f_{i}}$ be the maximal representation of $\ratg{f}{z}$. Then, we have $u(f(z))=(\q{0/a_{0}}:\cdots :\q{0/a_{n}})$. To show $u(f(z))=\Phi_{\la}(z)$, we only to show $\q{f_{n}(z)/f_{0}(z)}=\q{a_{0}/a_{n}}$ because both of $u(f(z))$ and $\Phi_{\la}(z)$ are points of the tropical line segment. We have
	\[ \ratg{f}{z}=\q{f+f(z)\cdot g}=\q{\sum_{i}(\frac{0}{f_{i}(y)}+\frac{f(z)}{f_{i}(x)})f_{i}}. \]
Let $b_{i}:=\q{0/f_{i}(y)+f(z)/f_{i}(x)}$. Since $\q{f_{0}(x)/f_{0}(y)}=f(y) \leq f(z)$, we have $b_{0}=\q{f(z)/f_{0}(x)}$. Then, we have $\ratg{f}{z}(y)=f(z)=\q{b_{0}\cdot f_{0}(y)}$. This implies $a_{0}=b_{0}$. Similarly, since $\q{f_{n}(x)/f_{n}(y)}=f(x) \geq f(z)$, we have $b_{n}=\q{0/f_{n}(y)}$. Then, we have $\ratg{f}{z}(x)=f(x)=\q{b_{n}\cdot f_{n}(x)}$. This implies $a_{n}=b_{n}$. For $0 \leq i<i \leq n$, we have
	\[ \q{f_{i}(z)/f_{i}(y)} \leq \q{f_{j}(z)/f_{j}(y)} \quad \text{and} \quad \q{f_{i}(x)/f_{i}(z)} \leq \q{f_{j}(x)/f_{j}(z)} \]
since $\Phi_{\la}(z)$ is a point of the tropical line segment.
Then, we have
	\[ f(z)=\q{\sum_{i}(f_{i}(z)/f_{i}(y))}=\q{f_{n}(z)/f_{n}(y)} \]
and
	\[ 0=g(z)=\q{\sum_{i}(f_{i}(z)/f_{i}(x))}=\q{f_{0}(z)/f_{0}(x)}. \]
Therefore, we have
	\[ \q{a_{0}/a_{n}}=\q{(f(z)/f_{0}(x))/(0/f_{n}(y))}=\q{f_{n}(z)/f_{0}(z)}. \qedhere\]
\end{proof}


Let $\la$ satisfy $(\ast)$. A point $p \in \ga$ is an \textit{indeterminacy point} of $\la$ if there exist two distinct elements $E,E' \in \la$ such that $p \in \supp(E) \cap \supp(E')$. The set of indeterminacy points is denoted by $I(\la)$.


\begin{lem} \label{zero}
Let $\la$ satisfy $(\ast)$, and let $x \in \ga$. We have $D_{x}(x) \geq D_{y}(x)$ for any $y \in \ga$.
\end{lem}

\begin{proof}
Let $f \in L_{\la}(D_{x})^{\ast}$ such that $D_{y}=D_{x}+\div(f)$. Then, $\div(f)(x) \leq 0$ since $f$ has the maximal value at $x$. Therefore, we have $D_{y}(x)=D_{x}(x)+\div(f)(x) \leq D_{x}(x)$.
\end{proof}

\begin{lem} \label{ind}
$I(\la)$ is a finite set.
\end{lem}

\begin{proof}
Let $p \in \ga$ be an indeterminacy point. We will show that $p$ is contained in $\supp (D_{y})$ for some $y \in \ga$ such that $\Phi_{\la}(y)$ is a leaf-end of $\Phi_{\la}(\ga)$. There exists $x \in \ga$ such that $D_{x}(p) \geq 1$ and $D_{x} \neq D_{p}$. Since $\Phi_{\la}(\ga)$ is a tree, there exists $y \in \ga$ such that $\Phi_{\la}(y)$ is a leaf-end of $\Phi_{\la}(\ga)$, and $\Phi_{\la}(x)$ is contained in the simple path between $\Phi_{\la}(p)$ and $\Phi_{\la}(y)$. Let $f \in L_{\la}(D_{p})^{\ast}$ such that $D_{y}=D_{p}+\div(f)$. By Lemma \ref{base}, we have $D_{x}=D_{p}+\div(\ratg{f}{x})$. Since $D_{p} \neq D_{x}$, we have $f(x)<f(p)$. Therefore, $\div(\ratg{f}{x})(p)=\div(f)(p)$, and thus $D_{y}(p)=D_{p}(p)+\div(f)(p)=D_{p}(p)+\div(\ratg{f}{x})(p)=D_{x}(p) \geq 1$. Since there are finitely many leaf-ends of $\Phi_{\la}(\ga)$, the set of indeterminacy points is finite.
\end{proof}

Let $f$ be a rational function on a tropical curve $\ga$. We define the ``non-constant locus'' of $f$, denoted by $\nc(f)$, as the set of all points $x$ of $\ga$ such that $f|_{U_{x}}$ is non-constant for any neighborhood $U_{x}$ of $x$.

\begin{lem} \label{ratg}
Let $\la$ satisfy $(\ast)$, let $x,y,z \in \ga$, and let $f \in L_{\la}(D_{x})^{\ast}$ such that $D_{y}=D_{x}+\div(f)$. If $z \in \nc(f)$, then we have $D_{z}=D_{x}+\div(\ratg{f}{z})$.
\end{lem}

\begin{proof}
By the continuities of $D_{z}$ and $D_{x}+\div(\ratg{f}{z})$, it is enough to show $D_{z}=D_{x}+\div(\ratg{f}{z})$ for almost all $z \in \nc(f)$. We may assume $z \notin I(\la) \cup \supp(D_{x})$. If $f(z)=f(x)$, then $\ord_{z}(f)<0$. This implies $z \in \supp(D_{x})$. Thus, we have $f(z)<f(x)$. Since $z \notin \supp(D_{x})$, we have $\ord_{z}(f)=D_{x}(z)+\div(f)(z)=D_{y}(z) \geq 0$. Since $f$ is non-constant around $z$, some of the outgoing slopes of $f$ at $z$ are positive. Thus, the support of $D_{x}+\div(\ratg{f}{z})$ has $z$. Since $z \notin I(\la)$, we have $D_{z}=D_{x}+\div(\ratg{f}{z})$.
\end{proof}

\begin{lem} \label{har}
Let $\la$ satisfy $(\ast)$, let $x \in \ga$, and let $h'$ be a half-edge of $\Phi_{\la}(x)$. Then, we have
	\[ \sum_{\substack{x \in h \\ h \mapsto h'}}\deg_{h}(\Phi_{\la})=D_{x}(x)-D_{y}(x), \]
where $y$ is a point of $\ga$ such that $\Phi_{\la}(x) \neq \Phi_{\la}(y)$ and $\Phi_{\la}(y)$ is contained in $h'$.
\end{lem}

\begin{proof}
Let $f \in L_{\la}(D_{x})^{\ast}$ such that $D_{x}+\div(f)=D_{y}$, and let $h$ be a half-edge of $x$. We may assume $f|_{h}$ is integral affine. Let us show that $h \mapsto h'$ if and only if the slope of $f|_{h}$ is non-zero. Suppose $h \mapsto h'$. Let $y'$ be a point contained in $h$ such that $\Phi_{\la}(y')=\Phi_{\la}(y)$. Since $D_{x} \neq D_{y}=D_{y'}$, we have $f(x)>f(y)=f(y')$. Therefore, the slope of $f|_{h}$ is non-zero. Conversely, suppose the slope of $f|_{h}$ is non-zero. For each point $z$ contained in $h$, we have $D_{z}=D_{x}+\div(\ratg{f}{z})$ by Lemma \ref{ratg}. Therefore, $\Phi_{\la}(z)$ is contained in the tropical line segment between $\Phi_{\la}(x)$ and $\Phi_{\la}(y)$. This implies $h \mapsto h'$.

For each half-edge $h$ of $x$ such that $h \mapsto h'$, the outgoing slope of $f|_{h}$ is equal to $-\deg_{h}(\Phi_{\la})$. Therefore, we have
	\[ \sum_{\substack{x \in h \\ h \mapsto h'}}\deg_{h}(\Phi_{\la})=-\ord_{x}(f)=D_{x}(x)-D_{y}(x).\qedhere \]
\end{proof}

By Lemma \ref{har}, $\Phi_{\la}$ is harmonic at each point except for indeterminacy points. Let us construct a modification $\wt{\ga}$ of $\ga$ and a harmonic map $\wt{\phi}:\wt{\ga} \to \Phi_{\la}(\ga)$. Let $p \in I(\la)$, and let $1 \leq n \leq D_{p}(p)$. We set
	\[ T_{p,n}:=\Phi_{\la}(\set{x \in \ga}{D_{x}(p)\geq n}). \]
Then, $T_{p,n}$ is a subtree of $\Phi_{\la}(\ga)$ or a singleton $\{\Phi_{\la}(p)\}$. We define $\wt{\ga}$ by attaching $T_{p,n}$ to $\ga$ in such a way that the point $\Phi_{\la}(p) \in T_{p,n}$ gets identified with the point $p \in \ga$, for all $p \in I(\la)$ and for all integer $1 \leq n \leq D_{p}(p)$. Let $\wt{\phi}:\wt{\ga} \to \Phi_{\la}(\ga)$ be a map defined by $\Phi_{\la}$ on $\ga$, and the inclusion on $T_{p,n}$.

\begin{prop}
$\wt{\phi}:\wt{\ga} \to \Phi_{\la}(\ga)$ is a finite harmonic morphism of degree $d$.
\end{prop}

\begin{proof}
Obviously, $\wt{\phi}$ is harmonic at $x \in \wt{\ga} \setminus \ga$. By Lemma \ref{har}, we only have to show that $\wt{\phi}$ is harmonic at indeterminacy points to prove harmonicity of $\wt{\phi}$. Let $p \in I(\la)$, let $h'$ be a half-edge of $\wt{\phi}(p)$, and let $y \in \ga$ such that $\wt{\phi}(p) \neq \wt{\phi}(y)$ and $\wt{\phi}(y)$ is contained in $h'$. By the construction of $\wt{\ga}$, the number of points $y' \in \wt{\ga} \setminus \ga$ such that $\wt{\phi}(y')=\wt{\phi}(y)$ is just $D_{y}(p)$. For each half-edge $h$ of $p$ contains such $y'$, we have $\deg_{h}(\wt{\phi})=1$. Therefore, we have
	\[ \sum_{\substack{p \in h \\ h \mapsto h'}}\deg_{h}(\wt{\phi})=D_{p}(p)-D_{y}(p)+D_{y}(p)=D_{p}(p). \]

Let us show $\deg(\wt{\phi})=d$. Let $x \in \ga$ such that $\wt{\phi}(x) \notin \wt{\phi}(I(\la))$. For $x' \in \ga$, if $\wt{\phi}(x')=\wt{\phi}(x)$, then $x' \notin I(\la)$ and thus, $\deg_{x'}(\wt{\phi})=D_{x'}(x')=D_{x}(x')$. For $x' \in \wt{\ga} \setminus \ga$, if $\wt{\phi}(x')=\wt{\phi}(x)$, then there exist the unique $p \in I(\la)$ of $D_{x}(p) \geq 1$ and the unique integer $1 \leq n \leq D_{x}(p)$ such that $x' \in T_{p,n}$. Conversely, for each $p \in I(\la)$ of $D_{x}(p) \geq 1$ and for each $1 \leq n \leq D_{x}(p)$, there exists the unique point $x' \in T_{p,n}$ such that $\wt{\phi}(x')=\wt{\phi}(x)$. For such $x' \in T_{p,n}$, we have $\deg_{x'}(\wt{\phi})=1$. Therefore, we have
\begin{align*}
	\deg(\wt{\phi})&=\sum_{\substack{x' \in \wt{\ga} \\ \wt{\phi}(x')=\wt{\phi}(x)}}\deg_{x'}(\wt{\phi}) \\
	&=\sum_{\substack{x' \in \ga \\ \wt{\phi}(x')=\wt{\phi}(x)}}\deg_{x'}(\wt{\phi})+\sum_{\substack{x' \in \wt{\ga} \setminus \ga \\ \wt{\phi}(x')=\wt{\phi}(x)}}\deg_{x'}(\wt{\phi}) \\
	&=\sum_{\substack{x' \in \ga \\ \wt{\phi}(x')=\wt{\phi}(x)}}D_{x}(x')+\sum_{\substack{p \in I(\la) \\ D_{x}(p) \geq 1}}D_{x}(p).
\end{align*}
For $y \in \ga$, if $y \in \supp(D_{x})$, then either $D_{y}=D_{x}$ or $y \in I(\la)$. Thus, we have $\deg(\wt{\phi})=\deg(D_{x})=d$.
\end{proof}

\begin{cor} \label{mthm}
Let $\ga$ be a tropical curve, and let $d$ be a positive integer. Then, the followings are equivalent:
\begin{enumerate}
	\item[$(1)$] $\ga$ is stably $d$-gonal.
	\item[$(2)$] $\ga$ has a liner system satisfying $(\ast)$.
\end{enumerate}
\end{cor}








\end{document}